\documentclass[12pt]{article}

\usepackage{amsmath, amsthm,amssymb}

\newtheorem{theorem}{Theorem}[section]
\newtheorem{lemma}[theorem]{Lemma}
\newtheorem{corollary}[theorem]{Corollary}

\renewcommand{\ge}{\geqslant}

\DeclareMathOperator{\size}{size}
\DeclareMathOperator{\rank}{rank}

\DeclareMathOperator{\orien}{o}

\newcommand{\ff}{\mathbb{F}}
\newcommand{\kk}{\mathbb{K}}
\newcommand{\pp}{\mathbb{P}}

\newcommand{\cc}{\mathbb{C}}
\newcommand{\rr}{\mathbb{R}}

\title{Topological classification of affine operators on unitary and Euclidean spaces}

\author {Tetiana Budnitska \\
Institute of Mathematics, Kyiv, Ukraine\\ budnitska\_t@ukr.net}

\date{}
\begin{document}

\maketitle

\begin{abstract}
We study affine operators on a unitary or Euclidean space $U$ up to topological conjugacy. An affine operator is a map $f: U \to U$ of the form $f(x)={\cal A}x+b$, in which ${\cal A}: U\to U$ is a linear operator and $b\in U$. Two affine operators $f$  and $g$ are said to be topologically conjugate if $g=h^{-1} f h$ for
some homeomorphism $h: U \rightarrow U$.

If an affine operator $f(x)={\cal A}x+b$ has a fixed point, then $f$ is topologically conjugate to its linear part $\cal A$. The problem of classifying linear operators up to topological conjugacy was studied by Kuiper and Robbin [Topological classification of linear endomorphisms, Invent. Math.~19 (no.\,2) (1973) 83--106] and other authors.

Let $f: U\to U$ be an affine operator without fixed point. We prove that $f$ is topologically conjugate to an affine operator $g: U \to U$ such that $U$ is an orthogonal direct sum of $g$-invariant subspaces $V$ and $W$,
\begin{itemize}
  \item the restriction $g|V$ of $g$ to $V$
is an affine operator that in some orthonormal basis of $V$ has the form
      \[(x_1, x_2, \dots, x_n) \mapsto (x_1+1, x_2, \dots, x_{n-1},\varepsilon x_n)\]
uniquely determined by $f$,
where $\varepsilon =1$ if $U$ is a unitary space, $\varepsilon =\pm 1$ if $U$ is a Euclidean space, and $n\ge 2$ if $\varepsilon =- 1$, and

  \item
the restriction $g|W$ of $g$ to $W$
is a linear operator that in some orthonormal basis of $W$ is  given by  a nilpotent Jordan matrix uniquely determined by $f$, up to permutation of blocks.
\end{itemize}
\medskip

\medskip


\end{abstract}

\section{Introduction}

We consider the problem of classifying affine operators on a unitary or Euclidean space $V$ up to topological conjugacy.
An
\emph{affine operator} $f:V\to V$ is a mapping of the form $f(x)={\cal A}x+b$, where ${\cal A}:V\to V$ is a linear operator and $b\in V$.

For simplicity, we always take $V=\mathbb F^n $ with $\ff=\cc$ or $\rr$ and the usual scalar product, then $f:\ff^n \to \ff^n$ has the form
\[
f(x)= Ax+b, \qquad A \in \ff^{\,n\times n}, \  b \in \ff^n.
\]

Two affine operators $f$, $g: \ff^n \rightarrow \ff^n$ are said to be \emph{conjugate} if there is a bijection $h: \ff^n \rightarrow \ff^n$ that transforms $f$ to $g$; that is,
\begin{equation}\label{jjj}
g=h^{-1} f h\qquad \text{(with respect to function composition)}.
\end{equation}
They are
\begin{itemize}
  \item[(a)] \emph{linearly conjugate} if $h$ in \eqref{jjj} is a linear operator;
  \item[(b)] \emph{affinely conjugate} if $h$ is an affine operator;
  \item[(c)]
  \emph{biregularly  conjugate} if $h$ is a \emph{biregular  map}, which means that $h$ and $h^{-1}$ have the form
\begin{equation}\label{rti}
(x_1,\dots,x_n) \mapsto (\varphi _1(x_1, \dots, x_n),\ldots,\varphi _n(x_1, \ldots, x_n)),
\end{equation}
in which all $\varphi _i$ are polynomials over $\mathbb F$;

  \item[(d)] \emph{topologically conjugate} if $h$ is a \emph{homeomorphism}, which means that $h$ and $h^{-1}$ are continuous and bijective.
\end{itemize}

Conjugations (a)--(c) are topological. Moreover,
\[ (a)\Rightarrow (b)\Rightarrow (c) \Rightarrow (d);
\]
that is,
linear conjugacy implies affine conjugacy implies biregular conjugacy implies topological conjugacy.

Let us
survey briefly known results on classifying affine operators up to conjugations (a)--(d):
\bigskip

(a) Each transformation of \emph{linear conjugacy} with $y=Ax+b$ corresponds to a change of the basis in $\ff^n$ and has the form
\begin{equation}\label{jst}
(A,b)\mapsto (S^{-1}AS,S^{-1}b),\qquad S\in\mathbb F^{n\times n}\text{ is nonsingular}.
\end{equation}
A canonical form of affine operators with respect to these transformations is easily constructed: if $\mathbb {F=C}$, then we can take $A$ in the Jordan canonical form and reduce $b$ by those transformations \eqref{jst} that preserve $A$; that is, by transformations $b\mapsto S^{-1}b$ for which  $S^{-1}AS=A$. Since $S$ commutes with the Jordan matrix $A$, it has the form described in \cite[Section VIII, \S 1]{Gantmacher}.
\bigskip

(b) Each transformation of \emph{affine conjugacy} corresponds to an affine change of the basis in $\ff^n$.
We say that an affine operator $x\mapsto Ax+b$ is \emph{nonsingular} if its matrix $A$ is nonsingular. Blanc~\cite{Blanc} proved that nonsingular affine operators $x\mapsto Ax+b$ and $x\mapsto Cx+d$ over an algebraically closed field of characteristic~$0$ are affinely conjugate if and only if their matrices $A$ and $C$ are \emph{similar}; i.e., $S^{-1}AS=C$ for some nonsingular $S$.
\bigskip

(c)
Blanc~\cite{Blanc} also obtained classification of nonsingular affine ope\-rators over an algebraically closed field $\mathbb K$ of characteristic~$0$ up to \emph{biregular conjugacy}:
\begin{itemize}
  \item two nonsingular affine operators over $\mathbb K$ with fixed points are biregularly conjugate if and only if their matrices are similar ($p$ is called a \emph{fixed point} of $f$ if $f(p)=p$);

  \item each nonsingular affine operator $f:\mathbb K^n\to \mathbb K^n$ without fixed point is biregularly conjugate to an ``almost-diagonal'' affine operator
      \begin{equation}\label{diag-autom}
      (x_1, x_2, \ldots, x_n)\mapsto (x_1+1, \lambda_2 x_2, \dots, \lambda_n x_n),
      \end{equation}
in which $1,\lambda_2, \dots, \lambda_n\in \mathbb K\setminus 0$ are all eigenvalues of the matrix of $f$ repeated according to their multiplicities. The affine operator \eqref{diag-autom} is uniquely determined by $f$, up to permutation of $\lambda_2,\dots,\lambda_n$.
\end{itemize}

(d) Affine operators on $\mathbb R^2$ were classified  up to \emph{topological conjugacy} by Ephr\"amowitsch \cite{Ephr}. In the present paper, we extend this classification to affine operators on $\mathbb R^n$ and $\mathbb C^n$. In Sections
\ref{topol-lin-R} and \ref{topola}, we classify affine operators of the following two types, respectively:
\bigskip

  \emph{Type 1: affine operators that have fixed point and have no eigenvalue being a root of $1$.} \
The problem of classifying affine operators with fixed point up to topological conjugacy is the problem of classifying all linear operators up to topological conjugacy. Indeed, each linear operator $x\mapsto Ax$  can be considered as the affine operator $x\mapsto Ax+0$ with the fixed point $x=0$. Conversely, if affine operators are considered up to topological conjugacy, then each $x\mapsto Ax+b$ with a fixed point can be replaced by its linear part $x\mapsto Ax$ since by Lemma \ref{pro-neryx-tochk-F-n} from Section \ref{topol-lin-R} they are topologically conjugate.

Kuiper and Robbin~\cite{Kuip-Robb, Robb} obtained a criterion of topological conjugacy of linear operators over $\rr$ without eigenvalues that are roots of~$1$. In Theorem~\ref{klas_m}, we recall their criterion, extend it to linear operators over $\mathbb C$, and give a canonical form for topological conjugacy of a linear operator over $\mathbb R$ and $\mathbb C$  without eigenvalues that are roots of $1$.

For simplicity, we do not consider linear operators with an eigenvalue being a root of $1$; the problem of topological classification of such operators was studied  by Kuiper and Robbin~\cite{Kuip-Robb, Robb}, Cappell and Shaneson~\cite{Capp-conexamp, Capp-2th-nas-n<=6, Capp-big-n<6, Capp-Contemp-n<6,Capp-n=6}, Hsiang and Pardon~\cite{Pardon}, Madsen and Rothenberg~\cite{Madsen}, and Schultz~\cite{Schultz}.
\medskip

\emph{Type 2: affine operators without fixed point.} \


In Theorem \ref{klik} we prove that each affine operator $f$ over $\mathbb{F=C}$ or $\mathbb R$ without fixed point is topologically conjugate to exactly one affine operator of the form
\begin{equation*}\label{htw-}
x\mapsto (I_k\oplus J_0)x + [1,0,\dots,0]^T
\end{equation*}
or, only if $\mathbb {F=R}$,
\begin{equation*}\label{htw1-}
x\mapsto (I_k\oplus [\,-1\,]\oplus J_0)x + [1,0,\dots,0]^T,
\end{equation*}
in which $k\ge 1$ and $J_0$ is a nilpotent Jordan matrix uniquely determined by $f$, up to permutations of blocks $(J_0$ is absent if $f$ is bijective$)$.

\bigskip

For each square matrix $A$ over $\mathbb {F\in\{C,R\}}$, there are a nonsingular matrix $A_{\textstyle *}$ and a nilpotent matrix $A_0$ over $\mathbb F$ such that
\begin{equation}\label{+-0a}
  A\quad\text{is similar to}\quad  A_{\textstyle *}\oplus A_0,
\end{equation}

We summarize criteria of topological conjugacy of affine operators in the following theorem.

\begin{theorem}\label{main teo}
Let $f(x)=Ax+b$ and  $g(x)=Cx+d$  be affine operators over $\mathbb F=\cc$ or $\rr$.
\begin{itemize}
  \item
Suppose that $f$ and $g$ have fixed points. Then $f$ and $g$ are topologically conjugate if and only if $x\mapsto Ax$ and  $x\mapsto Cx$  are topologically conjugate.

  \item
Suppose that $f$ has a fixed point and $g$ has no fixed point. Then $f$ and $g$ are not topologically conjugate.

  \item Suppose that $f$ and $g$ have no fixed points.
\begin{itemize}
  \item If $\mathbb{F=C}$ then $f$ and $g$ are topologically conjugate if and only if $A_0$ is similar to $B_0$.
  \item If $\mathbb{F=R}$ then $f$ and $g$ are topologically  conjugate if and only if the determinants of $A_{\textstyle *}$ and $C_{\textstyle *}$ have the same sign $($i.e., $\det (A_{\textstyle *}C_{\textstyle *})>0)$ and $A_0$ is similar to $C_0$.
\end{itemize}
\end{itemize}
\end{theorem}

\indent


\section{Affine operators with fixed point}\label{topol-lin-R}

In this section, we give a canonical form under topological conjugacy of an affine operator $f(x)=Ax+b$ that has a fixed point and whose matrix $A$ has no eigenvalue that is a root of unity.

We may, and will, consider only linear operators since
the following lemma reduces the problem of classifying affine operators with fixed point to the problem of classifying linear operators.

\begin{lemma}[\cite{st-1}]\label{pro-neryx-tochk-F-n}
An affine operator $f(x)=Ax+b$ over $\cc$ or $\rr$ is topologically conjugate to its linear part $f_{\text{\rm lin}}(x)=Ax$ if and only if $f$ has a fixed point. If $p$ is a fixed point of $f$, then
\begin{equation*}\label{ies}
f_{\text{\rm lin}}=h^{-1}fh,\qquad h(x):=x+p.
\end{equation*}
\end{lemma}

\begin{proof}
If $f(p)=p$, then $Ap+b=p$ and
\begin{align*}
(h^{-1}fh)(x)&= (h^{-1}f)(x+p)= h^{-1}(A(x+p)+b)\\&= h^{-1}(Ax+(p-b)+b)=h^{-1}(Ax+p)=Ax=f_{\text{\rm lin}}(x).
\end{align*}

Conversely, if $f$ and $ f_{\text{\rm lin}}$ are topologically conjugate, then $f$ and $f_{\text{\rm lin}}$ have the same number of fixed points. Since $f_{\text{\rm lin}}(0)=0$, $f$ has a fixed point too.
\end{proof}

For each $\lambda \in\mathbb C$, write
\begin{equation*}\label{fyi}
  J_n(\lambda ):=
\begin{bmatrix}
\lambda  &   &        & 0 \\
   1 & \lambda  &     &   \\
    &   \ddots  & \ddots &  \\
0   &     &     1   & \lambda
\end{bmatrix}\qquad \text{($n$-by-$n$)}.
\end{equation*}

For each $n\times n$ complex matrix $A=[a_{kl}+b_{kl}i]$, $a_{kl},b_{kl}\in\mathbb R$, we write
\begin{equation}\label{fff}
\overline A=[a_{kl}-b_{kl}i]
\end{equation}
and denote by $A^{\mathbb R}$ the \emph{realification} of $A$; that is, the $2n\times 2n$ real matrix obtained from $A$ by replacing each entry $a_{kl}+b_{kl}i$ with the block
\begin{equation}\label{fsy}
\begin{matrix}
 a_{kl} & -b_{kl}\\
  b_{kl} & a_{kl} \\
\end{matrix}
\end{equation}

Each square matrix $A$ over $\mathbb {F\in\{C,R\}}$ is similar to
\begin{equation}\label{+-0}
  A_0\oplus A_{01}\oplus A_1\oplus A_{1\infty},
\end{equation}
in which all eigenvalues $\lambda$ of $A_0$ (respectively, $A_{01}$, $A_1$, and $A_{1\infty}$) satisfy the condition
\[
\lambda =0\quad (\text{respectively, $0<|\lambda| <1,$ $|\lambda| =1,$ and $|\lambda| >1$}).
\]
Note that $A_0$ is the same as in~\eqref{+-0a} and $A_{01}\oplus A_1\oplus A_{1\infty}$ is similar to $A_{\textstyle *}$ in~\eqref{+-0a}.

In this section, we prove the following theorem; its part (a) in the case $\mathbb{F=R}$ was proved by Kuiper and Robbin~\cite{Kuip-Robb,Robb}.

\begin{theorem} \label{klas_m}

{\rm(a)} Let $f(x)=Ax$ and $g(x)=Bx$ be linear operators over $\mathbb {F=R}$ or $\mathbb C$ without eigenvalues that are roots of unity, and let $A_0,\dots,A_{1\infty}$ and $B_0,\dots,B_{1\infty}$ be constructed by $A$ and $B$ as in \eqref{+-0}.

\begin{itemize}
  \item[{\rm(i)}] If $\mathbb {F=R}$ then $f$ and $g$ are
 topologically conjugate if and only if
\begin{equation}\label{topol-matr-R}
\begin{matrix}
A_0 \text{ is similar to } B_0,\quad
\size A_{01} =\size B_{01}, \quad \det (A_{01} B_{01})>0,\\
A_1 \text{ is similar to } B_1,\quad
\size A_{1\infty} =\size B_{1\infty}, \quad \det (A_{1\infty} B_{1\infty})>0.
\end{matrix}
\end{equation}

  \item[{\rm(ii)}] If $\mathbb {F=C}$ then $f$ and $g$ are
 topologically conjugate if and only if  \begin{equation}\label{topol-matr-C}
\begin{matrix}
A_0 \text{ is similar to } B_0,\quad
\size A_{01} =\size B_{01}, \\
A_1\oplus \overline A_1 \text{ is similar to } B_1\oplus \overline B_1,\quad
\size A_{1\infty} =\size B_{1\infty}.
\end{matrix}
\end{equation}
\end{itemize}

{\rm(b)}
Each linear operator over $\mathbb {F=R}$ or $\mathbb C$ without eigenvalues that are roots of unity is topologically conjugate to a linear operator whose matrix is a direct sum that is uniquely determined up to permutation of summands and consists of
\begin{itemize}
  \item[{\rm(i)}] in the case $\mathbb {F=R}$:
\begin{itemize}
  \item
any number of summands
\begin{equation}\label{jdi}
J_k(0),\quad [\,1/2\,],\quad  J_k(\lambda )^{\mathbb R},\quad [\,2\,]
\end{equation}
$([\,1/2\,]$ and $[\,2\,]$ are the $1\times 1$ matrices with the entries $1/2$ and  $2),$
in which $\lambda$ is a complex number of modulus $1$ that is determined up to replacement by $\bar\lambda$ and that is not a root of unity,
  \item at most one $1\times 1$ summand $[\,-1/2\,]$, and
  \item at most one $1\times 1$ summand $[\,-2\,]$;
\end{itemize}

  \item[{\rm(ii)}] in the case $\mathbb {F=C}$:
\begin{equation}\label{zjdi}
J_k(0),\quad [\,1/2\,],\quad  J_k(\lambda ),\quad [\,2\,],
\end{equation}
in which $\lambda$ is a complex number of modulus $1$ that is determined up to replacement by $\bar\lambda$ and that is not a root of unity.
\end{itemize}
\end{theorem}

\begin{proof}
(a) The statement (i) was proved by Kuiper and Robbin~\cite{Kuip-Robb,Robb}. Let us prove (ii).

The abelian group $V=\mathbb {C}^n$ with respect to addition can be considered both as the $n$-dimensional vector space $V_{\mathbb C}$ over ${\mathbb C}$ and as the $2n$-dimensional vector space $V_{\mathbb R}$ over ${\mathbb R}$. Moreover, we can consider $V_{\mathbb C}$ as a unitary space with
the orthonormal basis
\begin{equation}\label{rwi1}
e_1=[1,0,\dots,0]^T,\ e_2=[0,1,\dots,0]^T,\ \dots,\
e_n=[0,0,\dots,1]^T,
\end{equation}
and $V_{\mathbb R}$ as a Euclidean space with the orthonormal basis
\begin{equation}\label{fow}
e_1,\ ie_1,\ e_2,\ ie_2,\ \dots,\ e_n,\ ie_n.
\end{equation}

For each
\begin{equation*}
v=(\alpha _1+\beta_1i)e_1+\dots+(\alpha _n+\beta_ni)e_n\in V,\quad \alpha _k,\beta_k\in\mathbb R,
\end{equation*}
its length in $V_{\mathbb C}$ and in $V_{\mathbb R}$ is the same:
\begin{equation*}
  |v|=(\alpha _1^2+\beta_1^2+\dots+\alpha _n^2+\beta_n^2)^{1/2}.
\end{equation*}
Thus,
\begin{equation}\label{4.26a}
\parbox{25em}
{a mapping $h:V\to V$ is a homeomorphism of $V_{\mathbb C}$ if and only if $h$ is a homeomorphism of $V_{\mathbb R}$.}
\end{equation}

Each linear operator $f:V_{\mathbb C}\to V_{\mathbb C}$ defines the linear operator $f^{\mathbb R}:V_{\mathbb R}\to V_{\mathbb R}$ ($f$ and $f^{\mathbb R}$ coincide as mappings on the abelian group $V$). By~\eqref{4.26a},
\begin{equation}\label{4.26}
\parbox{25em}
{two linear operators $f,g:V_{\mathbb C}\to V_{\mathbb C}$ are topologically conjugate if and only if $f^{\mathbb R},g^{\mathbb R}:V_{\mathbb R}\to V_{\mathbb R}$ are topologically conjugate.}
\end{equation}

Let $f(x)=Ax$ and $g(x)=Bx$ be linear operators on $V_{C}$ without eigenvalues that are roots of unity. Clearly, $A$ and $B$ are their matrices in the orthonormal basis
\eqref{rwi1}. Considering $f$ and $g$ as the linear operators $f^{\mathbb R}$ and $g^{\mathbb R}$ of $V_{\mathbb R}$, we find that the matrices of $f^{\mathbb R}$ and $g^{\mathbb R}$ in the basis \eqref{fow} are the realifications $A^{\mathbb R}$ and $B^{\mathbb R}$ of $A$ and $B$ (see \eqref{fsy}).

Since
\[
S^{-1}AS=A_0\oplus A_{01}\oplus A_1\oplus A_{1\infty}
\]
for some nonsingular $S$, we have
\begin{equation*}
  (S^{\mathbb R})^{-1}A^{\mathbb R}S^{\mathbb R}=A_0^{\mathbb R}\oplus A_{01}^{\mathbb R}\oplus A_1^{\mathbb R}\oplus A_{1\infty}^{\mathbb R}.
\end{equation*}
Analogously,
\[
B^{\mathbb R}\text{ is similar to } B_0^{\mathbb R}\oplus B_{01}^{\mathbb R}\oplus B_1^{\mathbb R}\oplus B_{1\infty}^{\mathbb R}.
\]

By \eqref{4.26} and the statement (i) of Theorem \ref{klas_m}(a), $f$ and $g$ are topologically conjugate if and only if $f^{\mathbb R}$ and $g^{\mathbb R}$ are topologically conjugate if and only if
\begin{equation}\label{tr-R}
\begin{matrix}
A_0^{\mathbb R} \text{ is similar to } B_0^{\mathbb R},\quad
\size A_{01}^{\mathbb R} =\size B_{01}^{\mathbb R}, \quad \det (A_{01}^{\mathbb R} B_{01}^{\mathbb R})>0,\\
A_1^{\mathbb R} \text{ is similar to } B_1^{\mathbb R},\quad
\size A_{1\infty}^{\mathbb R} =\size B_{1\infty}^{\mathbb R}, \quad \det (A_{1\infty}^{\mathbb R} B_{1\infty}^{\mathbb R})>0.
\end{matrix}
\end{equation}

For each complex matrix $M$, its realification $M^{\mathbb R}$ is similar to $M\oplus \overline M$ (see \eqref{fff}) because
\[
 \begin{bmatrix}
 1&1\\-i&i
 \end{bmatrix}^{-1}
 \begin{bmatrix}
 a&-b\\b&a
 \end{bmatrix}
\begin{bmatrix}
 1&1\\-i&i
 \end{bmatrix}=
 \begin{bmatrix}
 a+bi&0\\0&a-bi
 \end{bmatrix}.
\]
Since the Jordan canonical form of $A_0$ is a nilpotent Jordan matrix, $\overline A_0$ is similar to $A_0$. Thus, the condition ``$A_0^{\mathbb R}$ is similar to  $B_0^{\mathbb R}$'' is equivalent to the condition ``$A_0\oplus \overline A_0$ is similar to  $B_0\oplus \overline B_0$'' is equivalent to the condition ``$A_0$ is similar to  $B_0$''. The condition ``$\size A_{01}^{\mathbb R} =\size B_{01}^{\mathbb R}$'' is equivalent to the condition ``$\size A_{01}=\size B_{01}$''. The condition ``$\det (A_{01}^{\mathbb R} B_{01}^{\mathbb R})>0$'' always holds since
\[ \det (A_{01}^{\mathbb R} B_{01}^{\mathbb R})=\det (A_{01} B_{01})^{\mathbb R}=\det (A_{01} B_{01}\oplus \overline{A_{01} B_{01}})>0.
\]

We consider the remaining 3 conditions in \eqref{tr-R} analogously and get that \eqref{tr-R} is equivalent to \eqref{topol-matr-C}, which proves the statement (ii).
\bigskip

(b) This statement follows from (a) and the theorems about Jordan canonical form and real Jordan canonical form \cite[Theorems 3.1.11 and 3.4.5]{hor}
\end{proof}

\section{Affine operators without fixed points}\label{topola}

In this section, we prove the following theorem, which gives
a criterion of topological conjugacy and a canonical form under topological conjugacy for affine operators that have no fixed points.

\begin{theorem}\label{klik}
{\rm(a)} Let $f(x)=Ax+b$ and $g(x)= Cx+d$ be affine operators over $\mathbb{F=C}$ or $\mathbb R$ without fixed points. Let $A_{\textstyle *}, A_0$ and $C_{\textstyle *},C_0$ be constructed by $A$ and $C$ as in \eqref{+-0a}.

\begin{itemize}
  \item If $\mathbb{F=C}$ then $f$ and $g$ are topologically conjugate if and only if $A_0$ is similar to $B_0$.
  \item If $\mathbb{F=R}$ then $f$ and $g$ are topologically  conjugate if and only if the determinants of $A_{\textstyle *}$ and $C_{\textstyle *}$ have the same sign $($i.e., $\det (A_{\textstyle *}C_{\textstyle *})>0)$ and $A_0$ is similar to $C_0$.
\end{itemize}
{\rm(b)} Each affine operator $f$ over $\mathbb{F=C}$ or $\mathbb R$ without fixed point is topologically conjugate to exactly one affine operator of the form
\begin{equation}\label{htw}
x\mapsto (I_k\oplus J_0)x + [1,0,\dots,0]^T
\end{equation}
or, only if $\mathbb {F=R}$,
\begin{equation}\label{htw1}
x\mapsto (I_k\oplus [\,-1\,]\oplus J_0)x + [1,0,\dots,0]^T,
\end{equation}
in which $k\ge 1$ and $J_0$ is a nilpotent Jordan matrix determined by $f$ uniquely, up to permutations of blocks $(J_0$ is absent if $f$ is bijective$)$.
\end{theorem}

We give an affine operator $f(x)=Ax+b$ by the pair $(A,b)$ and write $f=(A,b)$.

For two affine operators $f:\mathbb F^{m}\to \mathbb F^{m}$ and $g:\mathbb F^{n}\to \mathbb F^{n},$ define the affine operator $f\oplus g:\mathbb F^{m+n}\to \mathbb F^{m+n}$ by
\[
(f\oplus g)(\begin{bmatrix}x \\y\end{bmatrix}) :=\begin{bmatrix}
                     f(x) \\
                     g(y) \\
                   \end{bmatrix};
\]
that is,
\begin{equation*}\label{rwi}
(A,b)\oplus (C,d)=\left(\begin{bmatrix}
                     A & 0 \\
                     0 & C \\
                   \end{bmatrix},
\begin{bmatrix}
                     b \\
                     d \\
                   \end{bmatrix}
 \right).
\end{equation*}
 We write $f\stackrel{\mathbb F}{\sim}g$ if $f$ and $g$ are topologically conjugate over $\mathbb F$.
Clearly,
\begin{equation}\label{xol}
f\stackrel{\mathbb F}{\sim} f'\text{ and } g\stackrel{\mathbb F}{\sim} g'\quad\Longrightarrow\quad f\oplus g\stackrel{\mathbb F}{\sim} f'\oplus g'.
\end{equation}

\subsection{Reduction to the canonical form}
\label{red}

In this section, we sequentially reduce an affine operator $y=Ax+b$ over $\mathbb{F=C}$ or $\mathbb R$ without fixed point by transformations of topological conjugacy to \eqref{htw} or \eqref{htw1}.
\bigskip

\emph{Step 1: reduce $y=Ax+b$ to  the form
\begin{equation}\label{twj}
\bigoplus_{i=1}^p(J_{m_i}(1),a_i)\oplus
\bigoplus_{i=p+1}^r(J_{m_i}(1),a_i)\oplus
(J_0,s)\oplus (B,c),
\end{equation}
in which $J_0$ is the Jordan canonical form of~$A_0$ $($see \eqref{+-0a}$)$, $1$ and $0$ are not eigenvalues of $B$, each of $a_1,\dots, a_p$ has a nonzero first coordinate, each of $a_{p+1},\dots, a_r$ has the zero first coordinate.}

We make this reduction by transformations of linear conjugacy \eqref{jst} over~$\mathbb F$.
\bigskip

\emph{Step 2: reduce \eqref{twj} to  the form
\begin{equation}\label{twk}
\bigoplus_{i=1}^p(J_{m_i}(1),a_i)\oplus
\bigoplus_{i=p+1}^r(J_{m_i}(1),0)\oplus
(J_0,0)\oplus (B,0),
\end{equation}
in which every $a_i$ has a nonzero  first coordinate.}

We make this reduction by using \eqref{xol} and the conjugations
\begin{equation}\label{fpm}
(J_{m}(1),a)\stackrel{\mathbb F}{\sim} (J_{m}(1),0),\quad
(J_0,s)\stackrel{\mathbb F}{\sim} (J_0,0),\quad
(B,c)\stackrel{\mathbb F}{\sim} (B,0),
\end{equation}
in which the first coordinate of $a$ is zero. The conjugations \eqref{fpm} hold by Lemma \ref{pro-neryx-tochk-F-n} since $
(J_{m}(1),a)$, $(J_0,s)$, and $(B,c)$ have fixed points (for example, $
(J_{m}(1),a)$ has a fixed point, which is a solution of the system  $J_{m}(1)x+a=x$; i.e., of the system $J_m(0)x=-a$).

Note that $p\ge 1$ since otherwise \eqref{twk} is a linear operator with the fixed point $0$, but $f$ has no fixed point.
\bigskip

\emph{Step 3: reduce \eqref{twk} to  the form
\begin{equation}\label{twy}
\bigoplus_{i=1}^p(J_{m_i}(1),e_1)\oplus (C,0)\oplus
(J_0,0),
\end{equation}
in which $e_1=[1,0,\dots,0]^T$ and
$
C:=\bigoplus_{i=p+1}^r J_{m_i}(1)\oplus B
$
is nonsingular.}

We use the conjugation
\begin{equation}\label{gai}
(J_m(1),a)\stackrel{\mathbb F}{\sim} (J_m(1),e_1),
\end{equation}
in which the first coordinate of $a$ is nonzero; that is, $a$ is represented in the form
\[
a=b[1,a_2,\dots,a_n]^T,\qquad b\ne 0.
\]
The conjugation \eqref{gai} is linear (see \eqref{jst}); it  holds since
\[(SJ_m(1)S^{-1},Se_1)= (J_m(1),a)
\] for
\[
S=b\begin{bmatrix}
     1 &&&&0  \\
     a_2 & 1 \\ a_3&a_2&1\\\ddots&\ddots&\ddots
&\ddots\\[-2mm]a_n&\ddots&a_3&a_2&1   \end{bmatrix}.
\]

\bigskip

\emph{Step 4: reduce \eqref{twy} to  the form
\begin{equation}\label{twyj}
\bigoplus_{i=1}^p(I_{m_i},e_1)\oplus (C,0)\oplus
(J_0,0).
\end{equation}}
We use the conjugation
\begin{equation}\label{gail}
(J_m(1),e_1)\stackrel{\mathbb F}{\sim} (I_m,e_1),
\end{equation}
which was constructed by Blanc \cite{Blanc}; he proved that
\[h  (J_m(1),e_1)=
(I_m,e_1)h ,\]
in which the homeomorphism $h: \mathbb F^{m}\to \mathbb F^{m}$ is biregular (see \eqref{rti}) and is
defined by
\[
h: (x_1,\dots,x_m)
        \mapsto
 (x_1,x_2+P_1,x_3+P_2,\dots,x_m+P_{m-1})
\]
with
\[
P_{k}:=(-1)^{k} \binom{x_1+k-1}{k+1}k+\sum^{k-1}_{i=1} (-1)^i\binom{x_1+i-1}{i}  x_{k+1-i}
\]
and
\[ \binom{\varphi }{r} :=\frac{\varphi(\varphi-1)(\varphi-2) \cdots (\varphi-r+1)}{r!}\quad \text{for each $\varphi\in\mathbb F[x_1]$}.
\]
\bigskip

\emph{Step 5: reduce \eqref{twyj} to  the form
\begin{equation}\label{twkd}
(I_1,[1])\oplus (D,0)\oplus
(J_0,0) ,
\end{equation}
in which $D:=I\oplus C$ is nonsingular.}

We use the conjugations
\[
\bigoplus_{i=1}^p(I_{m_i},e_1)\stackrel{\mathbb F}{\sim} (I_p,[1,\dots,1]^T)\oplus (I_q,0)\stackrel{\mathbb F}{\sim} (I_1,[1])\oplus(I_{q+p-1},0);
\]
the last conjugacy holds since $(I_2,[1,1]^T)\stackrel{\mathbb F}{\sim} (I_2,e_1)$, which follows from
\[
(S^{-1}I_2S,S^{-1}\begin{bmatrix}
     1 \\
     1 \\
   \end{bmatrix})=(I_2,e_1),\qquad S:=\begin{bmatrix}
     1 & 0 \\
     1 & 1 \\
   \end{bmatrix}\text{ (see \eqref{jst}).}
\]
\bigskip

\emph{Step 6: reduce \eqref{twkd} to  the form \eqref{htw} or \eqref{htw1}.}
In this step we consider two cases.
\bigskip

\emph{Case $\mathbb {F=R}.$}
For $\varepsilon =\pm 1$ and each nonsingular real $m\times m$ matrix $F$
that has an even number of Jordan blocks of each
size for every negative eigenvalue, we have the conjugation
\begin{equation}\label{dtk}
f\stackrel{\mathbb R}{\sim} g,\quad f:=(I_1,[1])\oplus(\varepsilon F,0),\quad g:=(I_1,[1])\oplus(\varepsilon I_m,0).
\end{equation}
Indeed, $ g=h^{-1}fh
$ for the
mapping $h: \mathbb R^{m+1}\to \mathbb R^{m+1}$ defined by
\begin{equation*}\label{kitf}
h: \begin{bmatrix}
     x \\
     y \\
   \end{bmatrix}\mapsto \begin{bmatrix}
     x \\
     \varepsilon F^{x}y \\
   \end{bmatrix},\qquad x\in\mathbb R,\ y\in\mathbb R^m
\end{equation*}
 since
\[
hg\begin{bmatrix}
     x \\
     y \\
   \end{bmatrix}= h\begin{bmatrix}
     x+1 \\
     \varepsilon y \\
   \end{bmatrix}= \begin{bmatrix}
     x+1 \\
     \varepsilon^2 F^{x+1}y \\
   \end{bmatrix}=f\begin{bmatrix}
     x \\
     \varepsilon F^{x}y \\
   \end{bmatrix}=fh\begin{bmatrix}
     x \\
     y \\
   \end{bmatrix}.
\]
The mapping $h$ is a homeomorphism since
\begin{itemize}
  \item
$h$ is continuous because the series
\begin{equation}\label{gpe}
F^{x}=e^{xG}=I+xG+\frac{(xG)^2}{2!} +\frac{(xG)^3}{3!}+\cdots
\end{equation}
has indefinite radius of convergence, where $G$ is a real matrix such that $F=e^G$ (it exists since by \cite[Theorem 6.4.15(c)]{h-J} for a real $M$ there is a real $N$ such that $M=e^N$ if and only if $M$ is nonsingular and has an even number of Jordan blocks of each size for every negative eigenvalue);

  \item
the inverse mapping
\begin{equation*}\label{kit}
h: \begin{bmatrix}
     x \\
     y \\
   \end{bmatrix}\mapsto \begin{bmatrix}
     x \\
     \varepsilon F^{-x}y \\
   \end{bmatrix},\qquad x\in\mathbb R,\ y\in\mathbb R^m
\end{equation*}
is continuous too.
\end{itemize}
This proves \eqref{dtk}.

\indent

Applying transformations of linear conjugation \eqref{jst} to \eqref{twkd}, we reduce $D$ to the form $P\oplus (-Q)$, in which $P$ is a nonsingular real $p\times p$ matrix without negative real eigenvalues, and $Q$ is a nonsingular real $q\times q$ matrix whose eigenvalues are positive real numbers. The affine operator \eqref{twkd} takes the form
\begin{equation*}\label{twkdw}
(I_1,[1])\oplus (P,0)\oplus (-Q,0)\oplus
(J_0,0);
\end{equation*}
by \eqref{xol} and \eqref{dtk}, it is topologically conjugate to
\begin{equation}\label{skdw}
(I_1,[1])\oplus (I_p,0)\oplus (-I_q,0)\oplus
(J_0,0).
\end{equation}

Taking $\varepsilon =1$ and $F=-I_2$ in \eqref{dtk}, we obtain
\begin{equation*}\label{rwk}
(I_1,[1])\oplus(-I_2,0) \stackrel{\mathbb R}{\sim} (I_3,e_1).
\end{equation*}
Applying this conjugation several times, we reduce \eqref{skdw} to the form \eqref{htw} or \eqref{htw1}.
We have proved that each affine operator over $\mathbb R$ without fixed point is topologically conjugate to \eqref{htw} or \eqref{htw1}.
\bigskip

\emph{Case $\mathbb {F=C}$.} Let us prove that
\begin{equation}\label{gswi}
f \stackrel{\mathbb C}{\sim} g,\quad f:=(I_1,[1])\oplus(D,0),\quad g:=(I_1,[1])\oplus(I_m,0),
\end{equation}
in which $D$ is the nonsingular complex $m\times m$ matrix from \eqref{twkd}.
Indeed, $ g=h^{-1}fh,
$ where
$h: \mathbb C^{m+1}\to \mathbb C^{m+1}$ is defined by
\begin{equation*}\label{kits}
h: \begin{bmatrix}
     x \\
     y \\
   \end{bmatrix}\mapsto \begin{bmatrix}
     x \\
     D^{x}y \\
   \end{bmatrix},\qquad x\in\mathbb C,\ y\in\mathbb C^m.
\end{equation*}
The mapping $h$ is a homeomorphism since
$D^{x}$ is represented in the form
\eqref{gpe} with $F:=D$ (the matrix $G$ exists since by \cite[Theorem 6.4.15(a)]{h-J} if $M$ is nonsingular then there is a complex $N$ such that $M=e^N$).

This proves \eqref{gswi}. Using it, reduce \eqref{twkd} to the form \eqref{htw}.
We have proved that each affine operator over $\mathbb C$ without fixed point is topologically conjugate to \eqref{htw}.

\subsection{Uniqueness of the canonical form}
\label{uniq}

In this section, we prove the uniqueness of the canonical form defined in Theorem~\ref{klik}(b).

Let $f$ and $g$ be two affine operators of the form \eqref{htw} or \eqref{htw1}; that is,
\[
f=f_{\textstyle *}\oplus f_0,\quad
f_{\textstyle *}=(I_{(\varepsilon)},e_1):{\mathbb F}^p\to{\mathbb F}^p,\quad
f_0=(J_0,0):{\mathbb F}^{n-p}\to{\mathbb F}^{n-p},
\]
and
\[
g=g_{\textstyle *}\oplus g_0,\quad
g_{\textstyle *}=(I_{(\delta)},e_1):{\mathbb F}^q\to{\mathbb F}^q,\quad
g_0=(J_0^{\prime},0):{\mathbb F}^{n-q}\to{\mathbb F}^{n-q},
\]
in which $\varepsilon ,\delta =\pm 1$,
\[I_{(1)}:=I,\quad I_{(-1)}:=I\oplus [\,-1\,],\]
and $J_0$ and $J_0^{\prime}$ are nilpotent Jordan matrices.
Let $f$ and $g$ be  topologically conjugate.

For each $i=1,2,\dots$, the images of $f^i$ and $g^i$ are the sets
\begin{equation*}\label{tpo}
V_i:= f^i{\mathbb F}^n={\mathbb F}^p\oplus J_0^i{\mathbb F}^{n-p},\qquad W_i:= g^i{\mathbb F}^n={\mathbb F}^q\oplus J_0^{\prime i}{\mathbb F}^{n-q},
\end{equation*}
and so they are vector subspaces of ${\mathbb F}^n$ of dimensions
\begin{equation}\label{suk}
\dim V_i=p+\rank J_0^i,\qquad \dim W_i=q+\rank J_0'^i.
\end{equation}

Since $f$ and $g$ are topologically conjugate, there exists a homeomorphism $h:\mathbb {F}^n\to \mathbb {F}^n$ such that
$hf=g h.$
Then
\begin{equation}\label{hyr}
hf^i= g^i h, \quad
hf^i\mathbb {F}^n= g^i h\,\mathbb{F}^n=g^i \mathbb {F}^n,\quad
h\,V_i=W_i.
\end{equation}
By \cite{h-w},  each two homeomorphic vector spaces have the same dimension; that is, the last equality implies
\begin{equation*}\label{gor}
\dim V_i=\dim W_i,\qquad i=1,2,\ldots
\end{equation*}

Fix any odd integer $m\ge\max(n-p,n-q)$. Then $J_0^m=J_0^{\prime m}=0$ and by~\eqref{suk}
\[
p=\dim V_m=\dim W_m=q.
\]

Thus, $f_{\textstyle *}=(I_{(\varepsilon)},e_1)$ and $g_{\textstyle *}=(I_{(\delta)},e_1)$ are affine bijections $V_{\textstyle *}\to V_{\textstyle *}$ on the same space
\begin{equation*}\label{syu}
V_{\textstyle *}:=V_m=W_m=\mathbb F^p.
\end{equation*}
By \eqref{hyr}, the restriction of $h$ to $V_{\textstyle *}$ gives some homeomorphism  $h_{\textstyle *}:V_{\textstyle *} \to V_{\textstyle *}$.
Restricting the equality $hf=g h$ to $V_{\textstyle *}$, we obtain
\begin{equation}\label{bij}
h_{\textstyle *} f_{\textstyle *}=g_{\textstyle *} h_{\textstyle *}.
\end{equation}
Therefore, $f_{\textstyle *}$ and $g_{\textstyle *}$ are topologically conjugate.

If $\mathbb {F=C}$, then $\varepsilon =\delta=1$.

Let $\mathbb {F=R}$. For each homeomorphism $\varphi$ on a Euclidean space, write $\orien (\varphi)= 1$ or $-1$ if it is orientation preserving or reversing. In particular, if $\varphi$ is a nonsingular affine operator $(A,b)$, then
\[
\orien (\varphi)=
  \begin{cases}
    1 & \text{if }\det A>0, \\
    -1 & \text{if }\det A<0.
  \end{cases}
\]
By \eqref{bij},
\[
\orien (h_{\textstyle *}f_{\textstyle *})=\orien (g_{\textstyle *} h_{\textstyle *}),\ \ \orien (h_{\textstyle *})\orien(f_{\textstyle *})=\orien (g_{\textstyle *})\orien( h_{\textstyle *}),\ \
\orien (h_{\textstyle *})\varepsilon =\delta \orien( h_{\textstyle *}),
\]
and so $\varepsilon =\delta$.

The nilpotent Jordan matrices
$J_0$ and $J_0'$ \emph{coincide up to permutation of blocks}
since by \eqref{suk}
the number of their Jordan blocks is equal to $n-\dim V_1$, the number of their Jordan blocks of size $\ge2$ is equal to $(n-\dim V_2)-(n-\dim V_1)$, the number of their Jordan blocks of size $\ge3$ is equal to $(n-\dim V_3)-(n-\dim V_2)$, and so on.

Thus, $\varepsilon =\delta$ and $f$ coincides with $g$ up to permutation of blocks in $J_0$ and~$J_0^{\prime}$.

\subsection{Conclusion}

Let  $f(x)=Ax+b$ be an affine operator over $\mathbb {F\in\{C,R\}}$.

We have showed in Sections \ref{red} and \ref{uniq} that $f$ is topologically conjugate to exactly one affine operator of the form  \eqref{htw} or \eqref{htw1}, which proves
the statement (b) of Theorem \ref{klik}.

Let $A_{\textstyle *}$ and $A_0$ be any nonsingular and nilpotent
parts of $A$ defined in~\eqref{+-0a}. Using the reduction of $f$ to
the canonical form described in Section \ref{red},
we find that
\begin{itemize}
  \item $f$ reduces to the form \eqref{htw} if $\mathbb {F=R}$ and $\det A_{\textstyle *}>0$, or if $\mathbb {F=C}$.

  \item  $f$ reduces to the form \eqref{htw1} if $\mathbb {F=R}$ and $\det A_{\textstyle *}<0$,
\end{itemize}
and $J_0$ in \eqref{htw} and \eqref{htw1} is the Jordan canonical form of $A_0$.
This proves the statement (a) of Theorem~\ref{klik}.

\begin{corollary}
An affine operator $f(x)=Ax+b$ over $\mathbb {C}$ and $\mathbb{R}$ has no fixed point if and only if it is linearly conjugate to an affine operator of the form
\begin{equation}\label{tso}
g(x)=(J_k(1)\oplus C)x+d,
\end{equation}
in which $d$ has a nonzero first coordinate. \end{corollary}

Indeed, \eqref{tso} has no fixed point since the first coordinates of $g(v)$ and $v$ are distinct for all $v$. Conversely, if $f(x)=Ax+b$ has no fixed point, then it is linearly conjugate to an affine operator of the form \eqref{twj}, in which $p\ge 1$ by Step 2.

\section*{Acknowledgments}

The author gratefully acknowledges the many helpful suggestions of S.I.~Maksymenko and V.V.~Sharko during the preparation of the paper.

\end{document}